\documentclass[12pt]{amsart}

\usepackage{amsfonts,amssymb,stmaryrd,amscd,amsmath,latexsym,amsbsy}

\newtheorem{theorem}{Theorem}[section]
\newtheorem{lemma}[theorem]{Lemma}

\newtheorem{corollary}[theorem]{Corollary}
\theoremstyle{definition}

\newtheorem{example}[theorem]{Example}
\newtheorem{question}[theorem]{Question}

\newtheorem{remark}[theorem]{Remark}

\newcommand{\ben}{\begin{enumerate}}
\newcommand{\een}{\end{enumerate}}

\theoremstyle{plain}

\newtheorem*{sol}{Solution}

\theoremstyle{definition}

\theoremstyle{remark}

\newcommand{\solu}[1]{\begin{sol}{\bf (\ref{#1})}}

\begin{document}

\title{A PI degree theorem for quantum deformations}

\author{Pavel Etingof}
\address{Department of Mathematics, Massachusetts Institute of Technology,
Cambridge, MA 02139, USA}
\email{etingof@math.mit.edu}

\maketitle 

\section{Introduction}

Let $F$ be an algebraically closed field. We show that if a quantum formal deformation $A$ of a commutative domain $A_0$ over $F$ is a PI algebra, then $A$ is commutative 
if ${\rm char}(F)=0$, and has PI degree a power of $p$ if ${\rm char}(F)=p>0$. 
This implies the same result for filtered deformations (i.e., filtered algebras $A$ such that ${\rm gr}(A)=A_0$). 

Note that a quantum formal deformation of a commutative domain $A_0$
may fail to be PI, even for finitely generated $A_0$ in characteristic $p$ (Example \ref{examp}(2)). However, we 
don't know if this is possible for filtered deformations. Thus we propose 

\begin{question}\label{que} Let ${\rm char}(F)=p>0$, and $A$ be a filtered deformation of a commutative finitely generated domain $A_0$ over $F$. Must $A$ be a PI algebra? In other words, must the division ring 
of quotients of $A$ be a central simple algebra? 
\end{question}

This question is closely related to the question asked in the introduction to \cite{CEW}, which would have affirmative answer if the answer to Question \ref{que} is affirmative. We don't know the answer to either of these questions even when $A_0$ is a polynomial algebra with generators in positive degrees. 
 
{\bf Acknowledgements.} The author is grateful to K. Brown, C. Walton and J. Zhang for useful discussions. The work of the author was partially supported by the NSF grant DMS-1502244. 

\section{Deformations of fields} 

Let $F$ be an algebraically closed field, and $A_0$ a field extension of $F$. 
Let $A$ be a quantum formal deformation of $A_0$ over $F[[\hbar]]$, i.e. 
an $F[[\hbar]]$-algebra isomorphic to $A_0[[\hbar]]$ as an $F[[\hbar]]$ module and equipped with an isomorphism of algebras $A/(\hbar)\cong A_0$
(for basics and notation on deformations, see \cite{EW}, Section 2).  

\begin{theorem}\label{PIdegr}  Suppose that $A$ is a PI algebra of degree $d$. 

(i) If ${\rm char}F=0$, then $d=1$ (i.e., $A$ is commutative). 

(ii) If ${\rm char}F=p>0$, then $d$ is a power of $p$. 
\end{theorem} 

\begin{proof}
Let $C$ be the center of $A$. It is easy to see that the division algebra of quotients of $A$ is $A[\hbar^{-1}]$ with center $C[\hbar^{-1}]$
(see \cite{EW}, Example 2.7). Moreover, by Posner's theorem (\cite{MR}, 13.6.5), $A[\hbar^{-1}]$ is a central division algebra over $C[\hbar^{-1}]$ of degree $d$, so $[A[\hbar^{-1}]:C[\hbar^{-1}]]=d^2$.  

Let $C_0=C/(\hbar)$. It is clear that $C_0$ is a subfield of $A_0$, and $C$ is a (commutative) formal deformation of $C_0$. 

\begin{lemma}\label{alge} $[A_0:C_0]=d^2$. 
\end{lemma} 

\begin{proof} Let $a_1^0,...,a_m^0\in A_0$ be linearly independent over $C_0$. 
Let $a_1,...,a_m$ be lifts of these elements to $A$. Then $a_1,...,a_m$ are linearly independent over $C$ and hence over $C[\hbar^{-1}]$. Thus $m\le d^2$. Moreover, if 
$a_1^0,...,a_m^0$ are a basis of $A_0$ over $C_0$ then $a_1,...,a_m$ 
are a free basis of $A$ over $C$ and hence a basis of $A[\hbar^{-1}]$ over $C[\hbar^{-1}]$, so 
$m=d^2$. 
\end{proof} 

Now for every integer $r\ge 0$, let $A_r\subset A_0$ be the field of all $x\in A_0$ 
which admit a lift to a central element of $A/(\hbar^{r+1})$. Note that $A_r\supset A_{r+1}$, and by Lemma \ref{alge}, this is a finite field extension.
 
Let us now prove (i). Assume the contrary, i.e. that $A$ is noncommutative. 
Let $r$ be the largest integer such that $[a,b]\in \hbar^rA$ for all $a,b\in A$. 
Then we have a nonzero Poisson bracket on $A_0$ given by $\lbrace{a_0,b_0\rbrace}=\hbar^{-r}[a,b]\text{ mod }\hbar$, where $a,b$ are any lifts of $a_0,b_0$ to $A$. 
Moreover, by definition $\lbrace,\rbrace$ is bilinear over $A_r$. 
Recall that $\lbrace,\rbrace$ is a derivation in each argument, and that any $K$-linear derivation
of a finite extension of a field $K$ 
of characteristic zero vanishes. Since $[A_0:A_r]<\infty$, this implies that 
$\lbrace,\rbrace=0$, a contradiction. This proves (i). 

We now prove (ii). 

\begin{lemma}\label{intersec} For large enough $r$,    
$A_r=C_0$. 
\end{lemma} 

\begin{proof}
For nonnegative integers $r\ge s$, let $C_{r,s}\subset  A/(\hbar^{s+1})$ be the set of elements liftable to a central element of 
$A/(\hbar^{r+1})$. It is clear that $C_{s,s}$ is the center of $A/(\hbar^{s+1})$, $C_{r,s}\supset C_{r+1,s}$, and $C_{r,s-1}$ is a quotient of $C_{r,s}$. Also $C_{r,s}$ is a $C_0/(\hbar^{s+1})$-submodule of
$A/(\hbar^{s+1})$. Let $C_{\infty,s}$ 
be the intersection of $C_{r,s}$ over all $r$. By Lemma \ref{alge}, 
$A/(\hbar^{s+1})$ has finite length as a $C_0/(\hbar^{s+1})$-module, 
so $C_{\infty,s}=C_{r(s),s}$ for a suitable $r(s)$. This implies that the natural map 
$C_{\infty,s}\to C_{\infty,s-1}$ is surjective (as it coincides with the map 
$C_{r,s}\to C_{r,s-1}$ for a suitable $r$). Let $C_{\infty,\infty}=\underleftarrow{\lim}C_{\infty,s}\subset A$. 

We claim that any element $a\in C_{\infty,\infty}$ is central in $A$. Indeed, $a$ projects to $a_s\in C_{\infty,s}\subset C_{s,s}$ which is central in 
$A/(\hbar^{s+1})$. Hence for any $b\in A$ we have $[a,b]=O(\hbar^{s+1})$. 
Since this holds for all $s$, we get that $[a,b]=0$. 

This implies that $C_{\infty,\infty}=C$ (as $C_{\infty,\infty}$ clearly contains $C$). Hence $C_{\infty,s}=C/(\hbar^{s+1})$ and in particular $C_{\infty,0}=C_0$. 
Hence $C_{r,0}=C_0$ for a large enough $r$. But by definition $C_{r,0}=A_r$, which implies the lemma. 
\end{proof}

\begin{lemma}\label{ppow} For all $r\ge 0$, one has $A_{r+1}\supset A_r^p$.  
\end{lemma} 

\begin{proof} Let $a_0\in A_r$, and $a$ be its lift to $A$ central modulo $\hbar^{r+1}$. Let $b\in A$. 
We have 
$$
[a^p,b]=\sum_{i=0}^{p-1} a^i[a,b]a^{p-1-i}=p[a,b]a^{p-1}+ \sum_{i=0}^{p-1} [a^i,[a,b]]a^{p-1-i}=
$$
$$
\sum_{i=0}^{p-1} [a^i,[a,b]]a^{p-1-i}
$$
(as we are in characteristic $p$). We have $[a,b]\in \hbar^{r+1} A$, hence $[a^i,[a,b]]\in \hbar^{r+2}A$. 
Thus $a^p\in A_{r+1}$.
\end{proof} 

Lemma \ref{ppow} implies that $A_r$ is a purely inseparable extension of $A_{r+1}$. In particular, 
$[A_r: A_{r+1}]$ is a power of $p$. 
Since by Lemma \ref{intersec} $A_r=C_0$ for large $r$, this implies that  $[A_0:C_0]$, and hence $d$, is  a power of $p$, as desired.  
\end{proof} 

\begin{remark} Here is another proof of Theorem \ref{PIdegr} (which deviates from the above proof after Lemma \ref{alge}). By Lemma \ref{alge}, 
it suffices to show that $A_0$ is a purely inseparable extension of $C_0$ 
(in particular, $A_0=C_0$ in characteristic zero). To this end, consider the algebra $B:= A\otimes_C A^{\rm op}$. Since $A[\hbar^{-1}]$ is a central division algebra of degree $d$ over $C[\hbar^{-1}]$, we have $B[\hbar^{-1}]\cong {\rm Mat}_d(C[\hbar^{-1}])$, hence $B$ does not contain nontrivial central idempotents. Therefore, the same holds for $B_0:=B/(\hbar)$ (otherwise 
we would have a nontrivial decomposition $B_0=B_0'\oplus B_0''$, which would lift to a decomposition $B=B'\oplus B''$, and $1_{B'}$ would be a nontrivial central idempotent in $B$). But 
$B_0=A_0\otimes_{C_0}A_0$. Hence $B_0$ has no nontrivial idempotents (i.e., is local). 
Let $x\in A_0$ be a separable element over $C_0$ and $K:=C_0[x]\subset A_0$. Then $K\otimes_{C_0}K\subset A_0\otimes_{C_0}A_0$ is reduced and projects onto $K$, hence 
contains nontrivial idempotents unless $K=C_0$. Hence $x\in C_0$, and $A_0$ is purely inseparable over $C_0$, as desired. 
\end{remark} 

\section{Deformations of domains}

Let us now extend Theorem \ref{PIdegr} to deformations of domains.

\begin{theorem}\label{PIdegr1} 
Theorem \ref{PIdegr} holds more generally, if $A_0$ is a domain over $F$.  
\end{theorem} 

\begin{proof} Following \cite{EW}, Subsection 2.2, let 
$Q(A)=\underleftarrow{\lim}Q(A/(\hbar^{N+1}))$, where   $Q(A/(\hbar^{N+1}))$ is the classical quotient ring of $A/(\hbar^{N+1})$.\footnote{The characteristic zero assumption of \cite{EW} is not used in these considerations.}  Also, let $Q_*(A)\subset Q(A)[\hbar^{-1}]$ be the quotient division algebra of $A$ (which exists since $A$ is a PI domain). Then $Q_*(A)$ is dense in $Q(A)[\hbar^{-1}]$ in the $\hbar$-adic topology 
(although in general $Q_*(A)\ne Q(A)$), and hence satisfies the same polynomial identities as $Q(A)[\hbar^{-1}]$. By Posner's theorem, $Q_*(A)$ is a central division algebra of degree $d$, hence so is $Q(A)[\hbar^{-1}]$ (as it is a division algebra satisfying the identities of $d\times d$ matrices but not matrices of smaller size). Also, $Q(A)$ is a formal quantum 
deformation of the quotient field $Q(A_0)$ of $A_0$, which is a field 
extension of $F$. Thus, Theorem \ref{PIdegr} applies to $Q(A)$, and the theorem is proved. 
\end{proof} 

\begin{corollary}\label{PIdegr2} 
Let $A$ be a $\Bbb Z_+$-filtered deformation of 
a commutative domain $A_0$ over $F$ (i.e., ${\rm gr}(A)=A_0$). 
Suppose that $A$ is a PI algebra of degree $d$. 

(i) If ${\rm char}F=0$, then $d=1$ (i.e., $A$ is commutative). 

(ii) If ${\rm char}F=p>0$ then $d$ is a power of $p$. 
\end{corollary} 

\begin{proof} Let $R(A)$ be the Rees algebra of $R$ and $\widehat{R}(A)$ the completed Rees algebra of $A$ (see e.g. \cite{EW}, Subsection 2.1). 
Then $R(A)$ satisfies the identities of matrices of size $d\times d$ but not smaller (since so does $A$).   
Since $R(A)$ is dense in $\widehat{R}(A)$ in the $\hbar$-adic topology, the same holds for 
$\widehat{R}(A)$. But $\widehat{R}(A)$ is a formal quantum deformation of $A$. 
Thus Theorem \ref{PIdegr1} implies the result. 
\end{proof} 

\begin{example}\label{examp} 1. Suppose ${\rm char} F=p>0$. Let $A$ be the formal $n$-th Weyl algebra, i.e. 
the $\hbar$-adically complete algebra over $F[[\hbar]]$ generated by 
$x_1,...,x_n,y_1,...,y_n$ with defining relations 
$$
[x_i,x_j]=[y_i,y_j]=0,\ [y_i,x_j]=\hbar \delta_{ij}.
$$
Then $A$ is a formal deformation of $A_0:=F[x_1,...,x_n,y_1,...,y_n]$, 
which is the completed Rees algebra of its filtered deformation (the usual Weyl algebra $\bold A_n(F)$). 
The center $A$ is $C=F[x_1^p,...,x_n^p,y_1^p,...,y_n^p][[\hbar]]$, so 
$A$ is PI of degree $d=p^n$. Note that if we have infinitely many generators $x_i,y_i$ 
then $A$ is not PI, so the ``finitely generated" assumption in Question \ref{que} is needed. 

2. A formal quantum deformation of a finitely generated commutative domain does not have to be PI, even in characteristic $p$. E.g., let $A$ be the formal quantum polynomial algebra, i.e. 
the $\hbar$-adically complete algebra generated by $x,y$ with relation $yx=(1+\hbar)xy$. 
This algebra is a quantum formal deformation of $A_0:=F[x,y]$. It has trivial center and hence is not PI. 
\end{example} 

\begin{remark} Here is a direct proof of Corollary \ref{PIdegr2}(i), bypassing localizations and formal deformations. 

Let $C$ be the center of $A$ and $C_0={\rm gr}(C)$.  
We claim that $A_0$ is algebraic over $C_0$. To show this, let $a_0\in A_0$ be 
a homogeneous element, and lift it to an element $a\in A$. Since $A$ is PI, by Posner's theorem 
it is algebraic over $C$, so there exists a nonzero $P\in C[t]$ such that $P(a)=0$. 
Taking the leading terms of this equation gives a nonzero polynomial 
$P_0\in C_0[t]$ such that $P_0(a_0)=0$, as desired. 

Now assume that $A$ is noncommutative. Let $\lbrace{,\rbrace}$ be the nonzero Poisson bracket 
on $A_0$ defined in the proof of Theorem \ref{PIdegr}(i). Given $a_0\in A_0$, the operator 
$\lbrace a_0,?\rbrace$ is a derivation of $A_0$ which vanishes on $C_0$. Since $A_0$ is algebraic over $C_0$ and ${\rm char}F=0$, this derivation vanishes, i.e., $\lbrace,\rbrace=0$, a contradiction. 

The same argument works for formal deformations (Theorem \ref{PIdegr1} when ${\rm char}F=0$). 
\end{remark}  

\begin{remark} Let us say that an algebra $A$ is locally PI if any finitely generated subalgebra of $A$ is PI. An example of such an algebra is the Weyl algebra $\bold A_{\mathcal I}(F)$  generated by $x_i,y_i$, $i\in {\mathcal{I}}$ for an infinite set ${\mathcal{I}}$ and ${\rm char}F=p>0$.  
Corollary \ref{PIdegr2} immediately implies that if $A$ is a locally PI filtered quantization 
of a commutative domain $A_0$ over $F$ then $A$ must be commutative 
if ${\rm char}(F)=0$, and the PI degree of every finitely generated subalgebra 
of $A$ is a power of $p$ if ${\rm char}(F)=p>0$. 
Thus, in the special case when $A$ is a connected Hopf algebra equipped with the coradical filtration and ${\rm char}F=0$, we recover \cite{BGZ}, Theorem 4.5. 
\end{remark}

\end{document}